\documentclass{amsart}
\usepackage{comment}
\title{Upper bounds for the number of solutions to quartic Thue equations}

 \author{Shabnam Akhtari}
\address{ \noindent CRM, University of Montreal,
 P.O. Box 6128,
Centre-ville Station,
Montreal
H3C 3J7}
 \email {akhtari@crm.umontreal.ca}
 

\begin{document}

%\maketitle

 \newtheorem{thm}{Theorem}[section]
\newtheorem{prop}[thm]{Proposition}
\newtheorem{lemma}[thm]{Lemma}
\newtheorem{cor}[thm]{Corollary}
\newtheorem{conj}[thm]{Conjecture} 

\begin{abstract}

We will give upper bounds for the number of integral solutions to quartic Thue equations.  Our main tool here is  a logarithmic curve  $\phi(x , y)$ that allows us to use
 the theory of linear forms in logarithms. This manuscript  improves the results of author's earlier work  with Okazaki \cite{AO5} by giving special treatments to  forms with respect to their signature.
\end{abstract}

%\begin{keyword} Thue Equation; Linear Forms in Logarithms
%\end{keyword}

%\begin{document}
\maketitle

%\maketitle

\section{Introduction}\label{Intro}

In this paper, we will study  binary quartic forms  with integer coefficients; i.e. polynomials of the shape
$$
F(x , y) = a_{0}x^{4} + a_{1}x^{3}y + a_{2}x^{2}y^{2} + a_{3}xy^{3} + a_{4}y^{4},
$$
with $a_{i} \in \mathbb{Z},$\,  $i \in \{0, 1 , 2 , 3 , 4\}$.
We aim to give upper bounds for the number of solutions to the equation 
\begin{equation}\label{1.25}
\left| F(x , y) \right| = 1.
\end{equation} 
Here we will count $(x , y)$ and $(-x , -y)$ as one solution.   Let $M(F)$ be the Mahler measure of $F(x , y)$. In \cite{AO5} we used some ideas of  Stewart \cite{Ste5} to bound the number of solutions with $|y| < M(F)^6$. We will slightly modify those ideas and use them to give an upper bound for the solutions of (\ref{1.25}) with $|y| < M(F)^{3.5}$.
Then we will improve the main result in   \cite{AO5} by giving better upper bounds for the number of   solutions $(x , y)$ with large $|y| \geq  M(F)^{3.5}$ to equation (\ref{1.25}).
The following  is the main result of this manuscript.
\begin{thm}\label{main5}
Let $F(x , y)$ be an irreducible   quartic binary form  with integer coefficients. %and Mahler measure $M(F)$. 
The Diophantine equation (\ref{1.25}) has at most $U_{F}$ (see the table below) solutions in integers $x$ and $y$, provided that the discriminant of $F$ is greater than $D_{0}$, where $D_{0}$ is an explicitly computable  constant.

\bigskip

\begin{tabular}{l @{ \qquad } l}
\hline
Signature of $F$ & $U_{F}$\\
%\multicolumn{2}{c}{$I_{F}$}\\
\hline
$(0 , 2)$ & $6$\\
$(2 , 1)$ & $14$ \\
$(4 , 0)$ & $26$ \\
\hline \\
\end{tabular}

\end{thm}

The reason for having different upper bounds for forms with different signature in Theorem \ref{main5}, relies  upon the fact that the number fields generated over $\mathbb{Q}$ by a root of the equation $F(x , 1) = 0$ have a rings integers with different numbers of fundamental units.

One can use the method of this manuscript to deal with particular quartic Thue equations, where  more information about the coefficients of the quartic form are available. Therefore, in applications, the strong condition on the size of discriminant may be removed.

The equation
$$
 x^4 -4x^3y - x^2y^2 + 4x y^3 + y^4 =1
$$
has exactly $8$ solutions $(x , y) = (0, 1), (1 , 0), (1 , 1), (-1 , 1), (4 , 1), (-1 , 4), (8 , 7), (-7 , 8)$ (see \cite{MPR5} for a proof). The author is not aware of any binary quartic forms $F(x , y)$ for which the equation $F(x , y) = 1$ has more than $8$ solutions. Magma \cite{Mag} did not find any solution to
$$
 x^4 -4x^3y - x^2y^2 + 4x y^3 + y^4 = -1.
$$

%Here, we are not concern with computational details. However, the method of this paper can be used to solve particular quartic Thue equations with sharper estimates for the heights and constants in lower bounds for linear forms in logarithms. Therefore, one may be able to remove the strong condition on the size of discriminant in applications. 

We will always assume that $F(x , y)$ is irreducible. In fact, when the form $F(x , y)$ is not irreducible over $\mathbb{Z}[x, y]$, we are in a  much simpler situation. In general, equation (\ref{1.25}) may have infinitely many integral solutions; $F(x , y)$ could, for instance, be a power of a linear or indefinite binary quadratic form that represents unity. If $F(x, y)$ is not an irreducible form, however, we may very easily derive a stronger version of our main theorem under the assumption that $F(x , y)$ is neither a $4$th power of a linear form nor a second power of a quadratic form.
Suppose that $F(x , y)$ is reducible and can be factored over $\mathbb{Z}[x , y]$ as follows
$$ 
F(x , y) = F_{1}(x , y) F_{2}(x , y),
$$
with $\textrm{deg}(F_{1}) \leq \textrm{deg}(F_{2})$ and $F_{1}$ irreducible over $\mathbb{Z}[x , y]$. Therefore, the following equations must be satisfied:
\begin{equation}\label{e1}
F_{1}(x , y) = \pm 1
\end{equation}
and  
\begin{equation}\label{e2}
F_{2}(x , y) = \pm 1.
\end{equation}
This means the number of solutions to (\ref{1.25}) is no more than the minimum of number of solutions to (\ref{e1}) and (\ref{e2}). 
First suppose that $F_{1}$ is a linear form. Then the equation (\ref{e2}) can be written as a polynomial of degree  $3$ in $x$ and therefore there are no more than $12$ complex solutions to above equations. 
 Now let us suppose that $F_{1}$ is a quadratic form. Using B\'ezout's theorem from classical  algebraic geometry concerning the number of common points of two plane algebraic curves, we conclude that (\ref{1.25}) has at most $16$ integral solutions.

In this manuscript, we give  new and sharper bounds for  the number of solutions to equation (\ref{1.25}).
The bound given in \cite{AO5} is improved here mostly due to  some adjustment in the definition of the logarithmic  curve $\phi(x , y)$ in Section \ref{TC}.  We also study the geometry of  binary forms with respect to their signature to get   amore precise understanding of the distribution of solutions to  (\ref{1.25}).
We also appeal to a result of Voutier (Proposition \ref{hM5}) to  estimate the height of algebraic numbers in the number field generated over $\mathbb{Q}$ by a root of the equation $F(x , 1) = 0$. These allow us to extend our method introduced in \cite{AO5} for quartic forms that split in $\mathbb{R}$ to all quartic forms.

%----------------------------------------------------------------------------
\section{Preliminaries} \label{Pre}
%-------------------------------------------------------------------------------

Let $f(x) = a_{n}x^{n} + \ldots + a_{1}x + a_{0}$ be the minimal polynomial of an algebraic number $\alpha \neq 0$. Suppose that over $\mathbb{C}$, 
$$
f(x) = a_{n}(x - \alpha_{1})\ldots (x - \alpha_{n}).
$$
We put 
$$
M(\alpha) = M(f) = \left| a_{n}\right| \prod_{i =1}^{n} \max(1 , \left| \alpha_{i}\right|).
$$
$M(\alpha)$ is known as the Mahler measure of $\alpha$.

The Mahler measure of a binary form $G(x , y)$ is defined to be equal to the Mahler measure of the polynomial $G(x , 1)$.   
In \cite{Mah5}, Mahler showed, for polynomial $G$ of degree $n$ and discriminant $D_{G}$, that
\begin{equation}\label{mahD5}
 M(G) \geq \left(\frac{D_{G}}{n^n}\right)^{\frac{1}{2n -2}}.
 \end{equation}

Let $F(x , y)$ be a quartic form that factors over $\mathbb{C}$ as follows
$$
F(x , y) = a_{0}(x - \alpha_{1} y) (x - \alpha_{2} y) (x - \alpha_{3} y)  (x - \alpha_{4} y).
$$
The discriminant $D$ of $F$ is given by
$$
D = D_{F}= a_{0}^{6} (\alpha_{1} - \alpha_{2})^{2}  (\alpha_{1} - \alpha_{3})^{2}   (\alpha_{1} - \alpha_{4})^{2}   (\alpha_{2} - \alpha_{3})^{2}  (\alpha_{2} - \alpha_{4})^{2}  (\alpha_{3} - \alpha_{4})^{2}.
$$
We call forms $F$ and $G$ equivalent if they are equivalent under $GL_{2} (\mathbb{Z})$-action; i.e. if there exist integers $a_{1}$ , $a_{2}$ , $a_{3}$ and $a_{4}$ such that 
$$
F(a_{1} x + a_{2}y , a_{3}x + a_{4}y) = G (x , y) 
$$
for all $x$, $y$, where $a_{1} a_{4} - a_{2} a_{3} = \pm 1$. 
We denote by $N_{F}$ the number of solutions in integers $x$ and $y$ of the Diophantine equation (\ref{1.25}). If $F$ and $G$ are equivalent  then $N_{F} = N_{G}$  and  $D_{F} = D_{G}$.

\bigskip

Suppose there is a solution $(x_{0} , y_{0})$ to  equation (\ref{1.25}). Since $$\gcd(x_{0} , y_{0}) = 1,$$ there exist integers $x_{1}$, $y_{1} \in \mathbb{Z}$ with 
$$
x_{0} y_{1} - x_{1} y_{0} =1.
$$
Then 
$$
F^{*}(1 , 0) = 1,
$$
where
$$
F^{*}(x , y) = F(x_{0} x + x_{1}y , y_{0}x + y_{1}y).
$$
Therefore, $F^{*}$ is a monic form equivalent to $F$.
From now on we will assume  $F$ is monic.

Let $\mathbb{Q}(\alpha_{1})^{\sigma}$ be the embeddings of the real number field $\mathbb{Q}(\alpha_{1})$ in $\mathbb{R}$,  $1 \leq \sigma \leq n$, where $\{\alpha_{1}, \alpha_{2} , \ldots, \alpha_{n} \}$ are roots of $F(x , 1) = 0$. We respectively have $n$ Archimedean valuations of $\mathbb{Q}(\alpha_{1})$:
$$
|\rho|_{\sigma} = \left|\rho^{(\sigma)}\right| , \   \   1 \leq \sigma \leq n.
$$
We enumerate simple ideals of $\mathbb{Q}(\alpha)$ by indices $\sigma > n$ and define
 non-Archimedean valuation of $\mathbb{Q}(\alpha)$ by the formulas
  $$
   |\rho | _{\sigma} = (  \textrm{Norm} \  \mathfrak{p})^{-k}, 
   $$ 
 where
 $$\
  k = \textrm{ord}_ {\mathfrak{p}} (\alpha) , \   \mathfrak{p} = \mathfrak{p}_{\sigma} , \   \sigma > n ,
  $$
 for any $\rho \in \mathbb{Q}^{*} (\alpha)$.  Then we have the \emph{product formula} :
 $$
  \prod_{1}^{\infty} |\rho|_{\sigma} = 1 ,  \   \rho \in  \mathbb{Q}(\alpha) .
  $$
  Note that $|\rho|_{\sigma} \neq 1$ for only finitely many $\rho$. We should also remark that if $\sigma_{2} = \bar{\sigma}_{1}$, i.e., 
$$\sigma_{2}(x) = \bar{\sigma}_{1}(x) \qquad \textrm{for} \qquad  x \in \mathbb{Q}(\alpha),
$$
then the valuations $|\, .\, |_{\sigma_{1}}$ and $|\, .\, |_{\sigma_{2}}$ are equal.
 We define the \emph{absolute logarithmic height} of $\rho$ as
 $$
 h(\rho) = \frac{1}{2n} \sum _{\sigma = 1}^{\infty}\left|\log |\rho|_{\sigma}\right| . 
 $$
  
\begin{prop}
For every non-zero algebraic number $\alpha$, we have $h(\alpha^{-1}) = h(\alpha)$. For algebraic numbers $\alpha_{1}, \ldots, \alpha_{n}$, we have 
$$
h(\alpha_{1} \ldots \alpha_{n}) \leq h(\alpha_{1})+ \ldots + h(\alpha_{n})
$$ 
and
$$
h(\alpha_{1} + \ldots + \alpha_{n}) \leq \log n + h(\alpha_{1})+ \ldots + h(\alpha_{n}).
$$
\end{prop}
\begin{proof}
See \cite{Bug5} for proof.
\end{proof}

\begin{prop}\label{hM5}(Voutier \cite{Vou5})
 Suppose $\alpha$ is a non-zero algebraic number of degree $n$ which is not a root of unity. If $n \geq 2$ then
 \begin{equation*}
 h(\alpha) = \frac{1}{n} \log M(\alpha)>\frac{1}{4} \left(\frac{\log \log n}{\log n}\right)^3.
 \end{equation*}
 \end{prop}
 
Let $\alpha$ and $\beta$ be two algebraic numbers. Then the following inequalities hold (see \cite{Bug5}):
\begin{equation}\label{bj15}
h(\alpha + \beta) \leq \log 2 + h (\alpha) + h(\beta)
\end{equation}
and
 \begin{equation}\label{bj25}
h(\alpha  \beta) \leq  h (\alpha) + h(\beta).
\end{equation}

 \begin{lemma}\label{mahl5}(Mahler \cite{Mah5})
 If $a$ and $b$ are distinct zeros of polynomial $P(x)$ with degree $n$, then we have
\begin{equation*}
  | a - b | \geq \sqrt{3} (n)^{-(n+2)/2} M(P)^{-n+1},
  \end{equation*}
 where $M(P)$ is the Mahler measure of $P$.
  \end{lemma}

    \begin{lemma}\label{esug}
 Let $f(x) = a_{n}x^{n} + \ldots + a_{1}x + a_{0}$ be an irreducible  polynomial of degree $n$ and $\alpha_{m}$ be one of its roots. For $f'(x)$ the derivative of $f$, we have
$$
2^{-(n-1)^2} \frac{\left|D_{f}\right|}{M(f)^{2n -2}}\leq    |f'(\alpha_{m})| \leq \frac{n(n+1)}{2} H(f) \left(\max (1 , |\alpha_{m}|)\right)^{n-1},
$$
where $D_{f}$ is the discriminant, $M(f)$ is the Mahler measure and $H(f)$ is the naive height of $f$.
 \end{lemma}
 \begin{proof}
The right hand side inequality is trivial by noticing that
$$
f'(x) = n a_{n}x^{n-1} + \ldots + a_{1}x.
$$
%$$
%|f'(\alpha_{m})|\leq n |f'(\alpha_{m})|
%$$
To see the left hand side inequality,  observe that for $\alpha_{i}$, $\alpha_{j}$, two distinct roots of $f(x)$, we have
 $$
 |\alpha_{i} - \alpha_{j}| \leq 2 \max(1 , |\alpha_{i}|) \max (1 , |\alpha_{j}|). 
 $$
Then
 \begin{eqnarray*}
 |f'(\alpha_{m})| & = & \prod_{i=1, i \neq m}^{n} |\alpha_{i} - \alpha_{m}| \geq \prod_{i=1, i \neq m}^{n} \frac{|\alpha_{i} - \alpha_{m}|}{\max(1 , |\alpha_{i}|) \max (1 ,
 |\alpha_{m}|)}\\
 & \geq & 2^{n-1 - n(n-1)} \prod _{j=1}^{n}\prod_{i=1, i \neq j}^{n} \frac{|\alpha_{i} - \alpha_{j}|}{\max(1 , |\alpha_{i}|) \max (1 ,
 |\alpha_{j}|)}\\
 & = & 2^{-(n-1)^2} \frac{\left|D_{F}\right|}{M(F)^{2n -2}}.
 \end{eqnarray*}
 \end{proof}

Suppose that $\mathbb{K}$ is an algebraic number field of degree $d$ over $\mathbb{Q}$ embedded in $\mathbb{C}$. If $\mathbb{K} \subset \mathbb{R}$, we put $\chi = 1$, and otherwise $\chi = 2$. We are given numbers $\alpha_{1} , \ldots \alpha_{n} \in \mathbb{K}^{*}$ with absolute logarithm heights $h(\alpha_{j})$, $1\leq j \leq n$. Let $\log \alpha_{1}$, $\ldots$, $\log \alpha_{n}$ be arbitrary fixed non-zero values of the logarithms. Suppose that 
$$
A_{j} \geq \max \{dh(\alpha_{j}) , |\log \alpha_{j}| \}, \  \   1 \leq j \leq n.
$$
Now consider the linear form 
$$
L = b_{1}\log\alpha_{1} + \ldots + b_{n}\log\alpha_{n},
$$
with $b_{1}, \ldots , b_{n} \in \mathbb{Z}$ and with the parameter 
$$B = max\{1 , \max\{b_{j}A_{j}/A_{n}: \  1\leq j  \leq n\}\}.
$$
 For brevity we put
$$
\Omega = A_{1} \ldots A_{n},
$$
$$
C(n) = C(n , \chi) = \frac{16}{n!\chi}e^{n}(2n + 1 + 2 \chi)(n + 2) (4n + 4)^{n + 1}\left(\frac{1}{2}en\right)^{\chi}  ,
$$
$$
C_{0} = \log (e^{4.4n + 7}n^{5.5}d^{2}\log (en)),
$$
$$
W_{0} = \log(1.5eBd\log(ed)).
$$
 The following is the main result of \cite{Mat25}.
 \begin{prop}[Matveev]\label{mat5}
If $\log\alpha_{1} , \ldots , \log\alpha_{n}$ are linearly independent over $\mathbb{Z}$ and
$b_{n} \neq 0$, then 
$$
\log|L| > -C(n) C_{0} W_{0}d^{2}\Omega.
$$
\end{prop}

%-------------------------------------------------------------------------------------------------------

\section{Summary of the Proof}\label{sp}
%---------------------------------------------------------------------------------------------

 Suppose that $(x , y)$ is an integral solution to equation (\ref{1.25}). Then we have
$$
(x - \alpha_{1} y) (x - \alpha_{2} y) (x - \alpha_{3}y) (x - \alpha_{4}y) = \pm 1.
$$
Therefore, for some $ 1 \leq i \leq 4$,
$$
\left| x - \alpha_{i} y\right| \leq 1.
$$
\textbf{ Definition}. We say the pair of solution $(x , y)$ is {\emph related} to $\alpha_{i}$ if
$$
\left| x - \alpha_{i} y\right| = \min_{1\leq j\leq 4} \left| x - \alpha_{j} y\right|.
 $$

\begin{prop}\label{P1}
Let $F(x , y)$ be an irreducible monic binary quartic form  with integer coefficients and Mahler measure $M(F)$. The Diophantine equation (\ref{1.25}) has at most $N_{1}$ (see the table below) solutions in integers $x$ and $y$ with $0< y<  M(F)^{3.5}$, provided that the discriminant of $F$ is greater than $D_{0}$, where $D_{0}$ is an explicitly computable  constant.

\bigskip

\begin{tabular}{l @{ \qquad } l}
\hline
Signature of $F$ & $N_{1}$\\
%\multicolumn{2}{c}{$I_{F}$}\\
\hline
$(0 , 2)$ & $5$\\
$(2 , 1)$ & $9$ \\
$(4 , 0)$ & $12$ \\
\hline \\
\end{tabular}
\end{prop}

Since $F(x , y)$ is monic, $(1 , 0)$ is a trivial solution to $F(x , y) = 1$. 
We will need to define a subset of solutions to (\ref{1.25}), called $\mathfrak{A}$ (see Section \ref{AA} for details). This  set contains the trivial solution $(1 , 0)$ and $5$ other pairs of solution   only when $F(x , 1) = 0$ has $4$ real roots.

\bigskip

\begin{tabular}{l @{ \qquad } l}
\hline
Signature of $F$ & $|\mathfrak{A}|$\\
\hline
$(0 , 2)$ & $1$\\
$(2 , 1)$ & $1$ \\
$(4 , 0)$ & $6$ \\
\hline \\
\end{tabular}

\begin{prop}\label{P3}
Let $F(x , y)$ be an irreducible  binary quartic form  with integer coefficients and Mahler measure $M(F)$. The Diophantine equation (\ref{1.25}) has at most $N_{2}$ (see the table below) solutions in integers $x$ and $y$ with $y\geq M(F)^{3.5}$, provided that the discriminant of $F$ is greater than $D_{0}$, where $D_{0}$ is an explicitly computable  constant.

\bigskip

\begin{tabular}{l @{ \qquad } l}
\hline
Signature of $F$ & $N_{2}$\\
\hline
$(0 , 2)$ & $0$\\
$(2 , 1)$ & $4$ \\
$(4 , 0)$ & $8$ \\
\hline \\
\end{tabular}
\end{prop}

%---------------------------------------------------------------------------------------
\section{Solutions with small $y$; the Proof of Proposition \ref{P1} }\label{p1}
%-----------------------------------------------------------------------------------------------------

We may suppose that $F(x , y)$ is a monic form and  has the smallest Mahler measure among all  monic forms that are equivalent to $F$. Assume that $F(x , 1) = 0$ has $r$ real roots and $2s$ non-real roots ($r + 2s = 4$).

Let $Y_{0}$ be a positive real number. Following Stewart \cite{Ste5} and Bombieri and Schmidt \cite{Bom5}, we will estimate the solutions $(x , y)$ to (\ref{1.25}) for which $0< y \leq Y_{0}$.    For binary form 
$$
F(x , y) = (x - \alpha_{1}y)\ldots  (x - \alpha_{n}y)
$$
put 
$$L_{i}(x , y) = x - \alpha_{i}y
$$
 for $i =1, \ldots, n$. Then
\begin{lemma}\label{S56}
Suppose $F$ is a monic binary form. Then for every solutions $(x , y)$ of (\ref{1.25}) we have
$$
\frac{1}{L_{i}(x , y)} - \frac{1}{L_{j}(x , y)} = (\beta_{j} - \beta_{i}) y,
$$
where $\beta_{1}$,\ldots, $\beta_{n}$ are such that the form
$$
J(u , w) = (u - \beta_{1}w)\ldots (u - \beta_{n}w)
$$
is equivalent to $F$.
\end{lemma}
\begin{proof}
 This is Lemma 4 of \cite{Ste5} and Lemma 3 of \cite{Bom5}, by taking $(x_{0}, y_{0}) $  equal to $ (1 , 0)$.
\end{proof}

For every  solution $(x , y) \neq (1 , 0)$ of (\ref{1.25}), fix $j = j(x , y)$ with
$$
\left|L_{j}(x , y) \right| \geq 1.
$$
Then, by Lemma \ref{S56},
\begin{equation}\label{S57}
\frac{1}{\left|L_{i}(x , y) \right|} \geq |\beta_{j} - \beta_{i}| |y| - 1.
\end{equation}
For complex conjugate  $\bar{\beta_{j}}$  of $\beta_{j}$, where $j = j(x , y)$,  we also have
$$  
\frac{1}{\left|L_{i}(x , y) \right|} \geq |\bar{\beta_{j}} - \beta_{i}| |y| - 1.
$$
Hence
$$
\frac{1}{\left|L_{i}(x , y) \right|} \geq |\textrm{Re}(\beta_{j}) - \beta_{i}| |y| - 1,
$$
where $\textrm{Re}(\beta_{j})$ is the real part of $\beta_{j}$.
We now choose an integer $m = m(x , y)$ with $|\textrm{Re}(\beta_{j}) - \beta_{j}|\leq 1/2$, and we obtain 
\begin{equation}\label{S58}
\frac{1}{\left|L_{i}(x , y) \right|} \geq \left(|m- \beta_{i}| -\frac{1}{2}\right) |y| - 1,
\end{equation}
for $i = 1,\ldots , n$.

\bigskip

\textbf{Definition}. For $1\leq i \leq n$, Let $\frak{X}_{i}$ be the set of solutions to (\ref{1.25}) with $1\leq y \leq Y_{0}$ and $\left|L_{i}(x , y) \right| \leq \frac{1}{2y}$.

\bigskip

\textbf{Remark } When $\alpha_{k}$ and $\alpha_{l}$ are complex conjugates, $\frak{X}_{l} = \frak{X}_{k}$ and therefore 
we only need to consider $r + s$ different sets $\frak{X}_{i}$.

\begin{lemma}\label{Sl5}
Suppose $(x_{1} , y_{1})$ and $(x_{2} , y_{2})$ are two distinct solutions in $\frak{X}_{i}$ with $y_{1} \leq y_{2}$. Then
$$
\frac{y_{2}}{y_{1}} \geq \frac{2}{7} \max(1 , |\beta_{i}(x_{1} , y_{1}) - m(x_{1} , y_{1})|).
$$
\end{lemma}
\begin{proof}
This is Lemma 5 of \cite{Ste5} and Lemma 4 of \cite{Bom5}.
\end{proof}

\begin{lemma}\label{Sl6}
Suppose $(x , y)$ is a solution to (\ref{1.25}) with $y > 0$ and $\left|L_{i}(x , y) \right|> \frac{1}{2y}$. Then 
$$
|m(x , y) - \beta_{i}(x , y)| \leq \frac{7}{2}.
$$
\end{lemma}
\begin{proof}
This is Lemma 6 of \cite{Ste5}.
\end{proof}

By Lemma \ref{S56} the form 
$$
J(u , w) = (u - \beta_{1}w)\ldots (u - \beta_{n}w)
$$
is equivalent to $F(x , y)$ and therefore the form
$$
\hat{J}(u , w) = (u - (\beta_{1}-m) w)\ldots (u - (\beta_{n}- m)w)
$$
is also equivalent to $F(x , y)$. Therefore, since we assumed that $F$ has the smallest Mahler measure among its equivalent forms, we get
\begin{equation}\label{Spre60}
\prod_{i=1}^{n} \max(1, |\beta_{1}(x , y)-m(x , y)|) \geq M(F). 
\end{equation}
For each set $\frak{X}_{i}$ that is not empty,  let  $(x^{(i)} , y^{(i)})$ be the element with the largest value of $y$.
Let $\frak{X}$ be the set of solutions of (\ref{1.25}) with $1 \leq y \leq Y_{0}$ minus the elements $(x^{(1)} , y^{(1)})$, \ldots, $(x^{(r+s)} , y^{(r+s)})$.
% Recall that there are at most $r+s$ of , for if $\alpha_{k}$ and $\alpha_{l}$ are complex conjugates and roots of $F($\left|L_{i}(x , y) \right| \leq \frac{1}{2y}x , 1) = 0$ then  %$(x^{(k)} , y^{(k)}) = (x^{(l)} , y^{(l)})$.
Suppose that, for integer $i$, the set   $\frak{X}_{i}$ is non-empty. Index the elements of $\frak{X}_{i}$ as 
$$(x_{1}^{(i)}, y_{1}^{(i)}), \ldots, (x_{v}^{(i)}, y_{v}^{(i)}),$$
 so that $y_{1}^{(i)} \leq \ldots \leq y_{v}^{(i)}$ (note that $(x_{v}^{(i)}, y_{v}^{(i)}) = (x^{(i)} , y^{(i)})$). By Lemma \ref{Sl5}
\begin{equation*}
\frac{2}{7} \max\left(1, \left|\beta_{i}(x_{k}^{(i)}, y_{k}^{(i)})\right|\right) \leq \frac{y_{k+1}^{(i)}}{y_{k}^{(i)}}
\end{equation*}
for $k = 1 \ldots, v-1$. Hence
$$
\prod_{(x , y) \in  \frak{X} \bigcap \frak{X_{i}} }\frac{2}{7} \max\left(1, \left|\beta_{i}(x_{k}^{(i)}, y_{k}^{(i)})\right|\right) \leq Y_{0}.
$$ 
%For  $i$ ($1 \leq i \leq n$),  if $(x , y) \in \frak{X_{i}} \bigcap \frak{X}$ then  $(x , y)$ is indexed in the set $\frak{X_{i}}$, in other words $(x , y) =  (x_{w}^{(i)}, y_{w}^{(i)})$ for some $ 1\leq w < v$ (see (\ref{ind})).
For $(x , y)$ in $\frak{X}$ but not in  $\frak{X_{i}}$ we have 
$$
\frac{2}{7} \max\left(1, \left|\beta_{i}(x_{k}^{(i)}, y_{k}^{(i)})\right|\right) \leq 1.
$$
By Lemma \ref{Sl6}. Thus
\begin{equation*}\label{S59}
\prod_{(x , y) \in  \frak{X}}\frac{2}{7} \max\left(1, \left|\beta_{i}(x_{k}^{(i)}, y_{k}^{(i)})\right|\right) \leq Y_{0}.
\end{equation*}
Comparing this with (\ref{Spre60}) and since we have at most $r+s$ different sets $\frak{X_{i}}$, we obtain 
\begin{equation}\label{S60}
\left( \left(\frac{2}{7}\right)^{4} M(F)\right)^{ |\frak{X}|} \leq Y_{0}^{r+s}.
 \end{equation}
 
 If $D_{F}$  satisfies the following numerical inequality
 $$
 D_{F} \geq 4^4 \, \left(\frac{7}{2}\right)^{4\times 6 \times 65}
 $$
 then we have
 $$
 \left(\frac{D_{F}}{4^4}\right)^{\frac{1}{6 \times 65}} \geq \left(\frac{7}{2}\right)^{4}.
 $$
 From here  and the fact that $M(F) \geq \left(\frac{D_{F}}{4^4}\right)^{\frac{1}{6}}$ (see   (\ref{mahD5})), we conclude that
 if the discriminant  is large enough then  $M(F)$ will be large enough to satisfy
 $$
\left(\frac{2}{7}\right)^{4} M(F) \geq M(F)^{64/65}.
$$
By (\ref{S60}),
\begin{equation}\label{sm5}
|\mathfrak{X}| < (r+s) \frac{65 \log Y_{0}}{ 64 \log M(F)}.
\end{equation}

When $F(x , y)$ has signature $(4 , 0)$,
choose $\theta_{1} > 0$ such that 
$$
\frac{65}{16} \left(\frac{11}{6} + \theta_{1}\right) < 8 .
$$
From (\ref{S60}), we conclude that in this case $|\mathfrak{X}|$ is at most $7$ and therefore (\ref{1.25}) has at most $11$ solutions with $1 \leq y < M(f) ^{\frac{11}{6} + \theta_{1}}$. 

\bigskip

When $F(x , y)$ has signature $(2 , 1)$
choose $\theta_{2} > 0$ such that 
$$
3\times \frac{65}{64} \left(\frac{11}{6} + \theta_{2}\right) < 6 .
$$
From (\ref{S60}), we conclude that in this case $|\mathfrak{X}|$ is at most $5$ and therefore (\ref{1.25}) has at most $8$ solutions with $1 \leq y < M(f) ^{\frac{11}{6} + \theta_{2}}$. 

\bigskip
We can repeat  the similar argument for forms with  signature $(0 , 2)$
and choose $\theta_{3} > 0$ such that 
$$
2\times \frac{65}{64} \left(\frac{11}{6} + \theta_{3}\right) < 4 .
$$
This will give us at most $5$ solutions with $1 \leq y < M(f) ^{\frac{11}{6} + \theta_{3}}$. But for this case, we have more to say in the next section.

\begin{lemma}\label{3S}
Let $F$ be a binary form of degree $n \geq 3$ with integer coefficients and  nonzero discriminant $D$. For every pair of integers $(x , y)$ with $ y \neq 0$
$$
\min_{\alpha} \left| \alpha - \frac{x}{y} \right| \leq \frac{2^{n-1} n^{n-1/2} \left(M(F)\right)^{n-2} |F(x , y)|}{|D(F)|^{1/2} |y|^n},
$$
where the minimum is taken over the zeros $\alpha$ of $F(z , 1)$.
\end{lemma}
\begin{proof}
This is Lemma 3 of \cite{Ste5}.
\end{proof}
\begin{lemma}\label{SC1}
Let $F(x , y)$ be a quartic binary form with   discriminant $D$, where $|D| \geq D_{0}$. Suppose that $\alpha_{i}$ is a root of $F(z , 1) = 0$. Suppose that $\theta > 0$.
Then related to $\alpha_{i}$,  there is at most $1$ solution for equation (\ref{1.25}) in integers $x$ and $y$ with $ \left(\frac{11}{6} + \theta \right) < y < M(F)^{3.5}$.
\end{lemma}
\begin{proof}
Assume that $(x_{1} , y_{1})$ and $(x_{2}, y_{2})$ are two distinct solutions  to (\ref{1.25}), both related to $\alpha_{i}$ with $ y_{2} > y_{1} > M(F)^2$. By Lemma \ref{3S},  we have
\begin{eqnarray*}
\left| \frac{x_{2}}{y_{2}} - \frac{x_{1}}{y_{1}} \right| &\leq &  \left| \alpha_{i} - \frac{x_{1}}{y_{1}} \right| +  \left| \alpha_{i} - \frac{x_{2}}{y_{2}} \right|\\ \nonumber
 &\leq& \frac{2^{3} 4^{7/2} \left(M(F)\right)^{2}}{|D(F)|^{1/2} |y_{1}|^4} +  \frac{2^{3} 4^{7/2} \left(M(F)\right)^{2}}{|D(F)|^{1/2} |y_{2}|^4}\\ \nonumber
 &\leq& \frac{2^{4} 4^{7/2} \left(M(F)\right)^{2}}{|D(F)|^{1/2} |y_{1}|^4}.
\end{eqnarray*}
Since $(x_{1} , y_{1})$, $(x_{2}, y_{2})$  are  distinct, we have  $|x_{2}y_{1} - x_{1}y_{2}|\geq 1$. Therefore,
$$
\left|\frac{1}{y_{1}y_{2}} \right|\leq \left| \frac{x_{2}}{y_{2}} - \frac{x_{1}}{y_{1}} \right| \leq \frac{M(F)^{2} }{ |y_{1}|^4}.
$$
This is because we assumed that $D_{F}$ is large. %and $|F(x_{j} , y_{j})|=1$. 
Thus,
\begin{equation}\label{S65}
\frac{y_{1}^{3}}{M(F)^{2}} \leq y_{2}.
\end{equation}
Following Stewart \cite{Ste5}, we define $\delta_{j}$, for $j = 1, 2$, by
$$
y_{j} = M(F)^{1+\delta_{j}}.
$$
By (\ref{mahD5}) the Mahler measure of $F$ is large  and  (\ref{S65}) implies that
$$
3 \delta_{1} \leq \delta_{2}.
$$ 
 From here, we conclude that 
$$
y_{2} > M(F)^{3.5}.
$$
In other words, related to each  root $\alpha_{i}$, there exists at most $1$ solution in $x$ and $y$ with $M(F)^ {\frac{11}{6} + \theta_{1}} < y < M(F)^{3.5}$.
\end{proof}

%---------------------------------------------------------------------------------------------------
\section{Proof of the Main Theorem for Forms with signature $(0 , 2)$}
%---------------------------------------------------------------------------------------------------

We will first show that  if a pair of integer $(x , y)$ satisfies
$F(x , y) = \pm 1$ and is related to a non-real root $\alpha$ of $F(x , 1) = 0$
 then
 $$
 |y| < M(F)^{9/4}.
 $$
\begin{lemma}\label{Gr}
For quartic binary form $F(x , y)$, let $\alpha$ be a non-real root of $F(x , 1) = 0$. If a pair of integer $(x , y)$ satisfies
$F(x , y) = \pm 1$ and is related to $\alpha$ then
\begin{equation}\label{AG}
|y| \leq  \frac{  2^{\frac{19}{4}}}{\left(\sqrt{3} \left|D_{F}\right|\right)^{1/4}} M(F)^{9/4}.
\end{equation}
\end{lemma}
\begin{proof}
 Let $\alpha = \mathfrak{r} +i \mathfrak{t}$, with $\mathfrak{t} \neq 0$, be a non-real root of $F(x, 1) = 0$. If a solution $(x , y)$ of (\ref{1.25}) is related to $\alpha$ then  $\bar{\alpha}$, the complex conjugate of $\alpha$, is also a root of  $F(x, 1) = 0$ and we have
 $$
 \left| \frac{x}{y} - \alpha \right| = \frac{\left| \frac{x}{y} - \alpha \right| + \left| \frac{x}{y} - \bar{\alpha} \right| }{2} \geq\frac{ \left| \alpha - \bar{\alpha} \right|}{2}.$$
 Moreover, if $\beta \neq \alpha$ is a root of $F(x , 1) = 0$ then
 $$
 \left| \frac{x}{y} - \beta \right| \geq\frac{ \left| \frac{x}{y} - \alpha \right| +  \left| \frac{x}{y} - \beta \right|}{2} \geq  \frac{\left| \beta - \alpha \right|}{2}.
 $$
 Thus 
     \begin{eqnarray*}
 \frac{1}{|y|^4} & = & \left| \frac{x}{y} - \alpha \right|  \prod_{\alpha_{i} \neq \alpha} \left| \frac{x}{y} - \alpha_{i}\right|  \\ \nonumber
 &\geq& \frac{\left|\alpha - \bar{\alpha}\right|}{2}  \prod_{\alpha_{i} \neq \alpha}\frac{ \left| \alpha - \alpha_{i}\right|}{2}  \\ \nonumber
& = &  \left|\alpha - \bar{\alpha}\right| \left|f' (\alpha)\right| 2^{-4}.
  \end{eqnarray*}
  By Lemma \ref{mahl5},
$$
  \left|\alpha - \bar{\alpha}\right| \geq \sqrt{3} (4)^{-3} M(F)^{-3}.
$$
This, together with Lemma \ref{esug}, shows that 
$$
 \frac{1}{|y|^4} \geq \sqrt{3} 2^{-19} \frac{\left|D_{F}\right|}{M(F)^{9}} .
$$
This completes our proof.
\end{proof}

Suppose that $(x_{1} , y_{1})$ and $(x_{2} , y_{2})$ are two solutions related to a fixed  root $\alpha$, with
$M(f) ^{\frac{17}{12} + \theta_{3}} < y_{1} < y_{2} < M(F)^{9/4}$.
Similar to the proof of Lemma \ref{SC1},
we define $\delta_{j}$, for $j = 1, 2$, by
$$
y_{j} = M(F)^{1+\delta_{j}}.
$$
Inequality  (\ref{S65}) implies that
$$
3 \delta_{1} \leq \delta_{2}.
$$ 
 From here, we conclude that if $y_{1} \geq M(F) ^{\frac{17}{12} + \theta_{3}}$ then 
 $$
y_{2} \geq M(F)^{9/4}.
$$
In other words, related to each  root $\alpha_{i}$, there exists at most $1$ solution in $x$ and $y$ with $M(F) ^{\frac{17}{12} + \theta_{3}} < y < M(F)^{9/4}$.

When $F(x , y)$ has signature $(0 , 2)$
Choose $\theta_{3} > 0$ such that 
$$
2\times \frac{65}{64} \left(\frac{17}{12} + \theta_{3}\right) < 3 .
$$
From (\ref{S60}), we conclude that in this case $|\mathfrak{X}|$ is at most $2$ and therefore (\ref{1.25}) has at most $4$ solutions with $1 \leq y < M(f) ^{\frac{17}{12} + \theta_{3}}$. 

Since $F$ is monic, $(1 , 0)$ is a trivial solution to  equation (\ref{1.25}). Therefore, when $F$ has signature $(0 , 2)$ the number of solutions to (\ref{1.25}) does not exceed $6$.

%------------------------------------------------------------------------------------------------------------------------------------------
\section{Transcendental Curve $\phi(x , y)$ }\label{TC}
%------------------------------------------------------------------------------------------------------------------------------------------

Fix a positive integer $k$. Define
\begin{equation}\label{fi5}
 \phi_{m}(x , y) =  \log \left| \frac{D^{\frac{1}{4k}}(x - y\alpha_{m})}{\left| f'(\alpha_{m})\right|^{\frac{1}{k} }}\right|
\end{equation}
and
\begin{equation}\label{dfi}
\phi(x , y) = \left( \phi_{1}(x , y) , \phi_{2}(x , y), \phi_{3}(x , y) , \phi_{4}(x , y) \right).
\end{equation}
Let $\| \phi(x , y) \|$ be the $L_{2}$ norm of the vector $\phi(x , y)$.

\bigskip

\textbf{Remark}.
In \cite{AO5},  a logarithmic curve $\phi(x , y)$  is defined by taking $k = 3$. The new general definition of $\phi(x , y)$ in this paper gives us the freedom of choosing $k$ large enough to  make our approximations  sharper. In order to have our estimations correct it is sufficient to take $k = 90$.

\begin{lemma}\label{lem15}
Suppose that $(x ,y)$ is a solution to the equation $F(x , y) = 1$, where $F$ is  the binary form in Theorem $\ref{main5}$. If
$$
\left| x - \alpha_{i} y\right| = \min_{1\leq j\leq 4} \left| x - \alpha_{j} y\right|
 $$
then 
$$
 \left\| \phi(x , y) \right\| \leq 6\, \log \frac{1}{ \left| x - \alpha_{i} y\right|} +  \left\| \phi(1 , 0) \right\| . $$
\end{lemma}
\begin{proof}
Let us assume that 
$$
\left| x - \alpha_{s_{j}} y\right| < 1, \qquad  \textrm{for}  \  1 \leq j \leq p
 $$
 and 
 $$
\left| x - \alpha_{b_{k}} y\right| \geq 1, \qquad \textrm{for} \ 1\leq k \leq 4-p,
$$
where $1\leq p, s_{j} , b_{k} \leq 4$.
We have 
$$
\prod_{k} \left| x - \alpha_{b_{k}} y\right| = \frac{1} {\prod_{j} \left| x - \alpha_{s_{j}} y\right|} .
$$
Therefore, for any $ 1 \leq k \leq 4-p$, we have
$$
\log \left| x - \alpha_{b_{k}} y\right| \leq p \log  \frac{1}{\left| x - \alpha_{i} y\right|}.
 $$
 Since 
  $$
\left| x - \alpha_{i} y\right| = \min_{1\leq j\leq 4} \left| x - \alpha_{j} y\right|,
 $$
we also have
  $$
\left|\log \left| x - \alpha_{s_{j}} y\right| \right| \leq \left| \log  \left| x - \alpha_{i} y\right| \right|. 
 $$ 
From here, we conclude that
  \begin{eqnarray*}
 \left\| \phi(x , y) \right\| & \leq& \left\| \phi(1 , 0) \right\|  +(4 -p)p \left| \log  \left| x - \alpha_{i} y\right| \right| +p\left| \log  \left| x - \alpha_{i} y\right| \right|  \\ \nonumber
  & = &  \left\| \phi(1 , 0) \right\| +
(5p - p^2) \left| \log  \left| x - \alpha_{i} y\right| \right| .
 \end{eqnarray*}
  The function $A(p) = 5p - p^2$ obtains its maximum value $6$ over  $p \in \{1, 2, 3, 4\}$.
    \end{proof}
    
In the following lemma  we approximate the size of $f'(\alpha)$ in terms of the discriminant and heights of $f$ , where $f'$ is the derivative of the polynomial $f$ and $\alpha$ is a root of $f = 0$.

 \begin{lemma}\label{lem10}
Suppose that $F$ is a monic quartic  binary form.
Then $(1 , 0)$ is a solution to the equation $\left|F(x , y)\right| = 1$ and 
$$
 \left\| \phi(1 , 0) \right\| \leq 4 \log \left(2^{9/k} |D|^{\frac{-3}{4k}} M(F)^{\frac{6}{k}}\right),
 $$
\end{lemma}
\begin{proof}
By the definition of $\phi(x , y)$,
$$
\left\| \phi(1 , 0) \right\| \leq \sum_{m=1}^{4} \log \left| \frac{D^{\frac{1}{4k}}}{\left| f'(\alpha_{m})\right|^{\frac{1}{k} }}\right|
$$
Substituting the lower bound for the   $\left| f'(\alpha_{m})\right|^{\frac{1}{k} }$ from Lemma \ref{esug} completes the proof. 
\end{proof}
   %%%%%%%%%%%%%%%%%%%%%%%%%%%%Granville's suggestion
   
    %---------------------------------------------------------------------------------------------------------------------------------------
\section{Exponential Gap Principle}\label{sec25} 
%-----------------------------------------------------------------------------------------------------------------------------------------

Here, our goal is to show
\begin{prop}\label{exg5}
Suppose that  $(x_{1} , y_{1})$, $(x_{2} , y_{2})$ and $(x_{3} , y_{3})$ are three pairs of non-trivial solutions to (\ref{1.25}) with 
$$
\left| x_{j} - \alpha_{4} y_{j}\right| < 1,
$$
and $|y_{j}| >M(F)^{3.5} $,
for $j \in \{1 , 2 , 3\}$.  If $\left\| \phi(x_{1} , y_{1}) \right\| \leq \left\| \phi(x_{2} , y_{2}) \right\| \leq \left\| \phi(x_{3} , y_{3}) \right\|$ then
$$
\left\| \phi(x_{3} , y_{3}) \right\| > 0.00014 \exp\left(\frac{\left\| \phi(x_{1} , y_{1}) \right\|}{6}\right).
$$
\end{prop}
In \cite{AO5}, we showed that when $F(x , y)$ splits in $\mathbb{R}$, i.e. when the signature is $(4 , 0)$,  assuming that $y_{3} > y_{1} > M(F)^{6}$, one can get the following inequality
$$
\left\| \phi(x_{3} , y_{3}) \right\| > \exp\left(\frac{\left\| \phi(x_{1} , y_{1}) \right\|}{6}\right)2 \, \sqrt{3} \log^4 \frac{1 + \sqrt{5}}{2},$$
that is sharper than the inequality in Proposition \ref{exg5}. This makes the value of $D_{0}$ in our main theorem smaller and does not decrease the upper bound for the number of solutions.

Observe that  for three pairs of solutions in Proposition \ref{exg5},  the three points $\phi(x_{1}, y_{1})$,  $\phi(x_{2}, y_{2})$ and  $\phi(x_{3}, y_{3})$ form a triangle $\Delta$. To establish Proposition \ref{exg5}, we will find a lower bound and an upper bound for the area of $\Delta$. Then comparing these bounds, Proposition \ref{exg5} will be proven. The length of each side of $\Delta$ is less than $2\left\| \phi(x_{3} , y_{3}) \right\|$. Lemma \ref{gp15} gives an upper bound for the height of $\Delta$. 
 Let $(x , y) \neq (1 , 0)$ be a solution to (\ref{1.25}) and let $t = \frac{x}{y}$. We have
$$
\phi(x , y) = \phi(t) = \sum_{i=1}^{4} \log \frac{|t - \alpha_{i}|}{\left| f'(\alpha_{i})\right|^{\frac{1}{k}}} {\bf{b_{i}}},
$$
where,
$$
{\bf{b_{1}}} = \frac{1}{4}(3, -1, -1, -1) , \qquad  {\bf{b_{2}}} = \frac{1}{4}(-1, 3, -1, -1),
$$
$$
{\bf{b_{3}}} = \frac{1}{4}(-1, -1,3, -1) , \qquad  {\bf{b_{4}}} = \frac{1}{4}(-1,  -1, -1, 3),
$$
Without loss of generality, we will assume  that $\alpha_{4}$ is a real root and for the pair of solution $(x , y)$ we have 
$$
\left| x - \alpha_{4} y\right| \leq 1.
$$
We may write (see the definition of $\phi(x , y)$ in (\ref{dfi}))
\begin{equation}\label{E5}
\phi(x , y) = \phi(t) = \sum_{i=1}^{3} \log \frac{|t - \alpha_{i}|}{\left| f'(\alpha_{i})\right|^{\frac{1}{k}}} {\bf{c_{i}} }+ E_{4} {\bf{b_{4}}},
\end{equation}
where, for $1\leq i \leq 3$,
$$
{\bf{c_{i}}} = {\bf{b_{i}}} + \frac{1}{3} {\bf{b_{4}}} , \quad  E_{4} = \log \frac{\left| t - \alpha_{4} \right|}{\left|f'(\alpha_{4})\right|^{\frac{1}{k}}} - \frac{1}{3} \sum_{i=1}^{3} \log \frac{|t - \alpha_{i}|}{\left| f'(\alpha_{i})\right|^{\frac{1}{k}}} 
$$
One can easily observe that
$$
{\bf{c_{i}}} \perp {\bf{b_{4}}}, \  \textrm{for}\   1\leq i \leq 4.
$$

\begin{lemma}\label{gp15}
Let
$$
{\bf{L_{4}}} = \sum_{i=1}^{3} \log \frac{|\alpha_{4} - \alpha_{i}|}{\left| f'(\alpha_{i})\right|^{\frac{1}{2}}} {\bf{c_{i}}} + z {\bf{b_{4}}} , \quad z \in \mathbb{R}.
$$
Suppose that $(x , y) \neq (1 , 0)$ is a  pair of solution to (\ref{1.25}) with
$$
\left| x - \alpha_{4} y\right| = \min_{1\leq j\leq 4} \left| x - \alpha_{j} y\right|
 $$ 
and $y \geq M(F)^{3.5}$.
Then the distance between $\phi(x, y)$ and the line $\bf{L_{4}}$ is less than 
$$
 \exp\left(\frac{-\left\| \phi(x , y) \right\|}{6}\right).
$$
\end{lemma}
\begin{proof}
The proof goes exactly as the proof of Lemma 6.3 in \cite{AO5}.
\end{proof}

Lemma \ref{gp15} shows that the height of $\Delta$ is at most 
$$
2\, \exp\left(\frac{-\left\| \phi(x_{1} , y_{1}) \right\|}{6}\right).
$$
Therefore, the area of $\Delta$ is less than 
\begin{equation}\label{up5}
2   \left\| \phi(x_{3} , y_{3}) \right\| \exp\left(\frac{-\left\| \phi(x_{1} , y_{1}) \right\|}{6}\right).
\end{equation}

Let us now estimate the area of $\Delta$ from below. Since
$$
F(x , y) = (x - \alpha_{1} y) (x - \alpha_{2} y)(x - \alpha_{3}y) (x - \alpha_{4}y) = \pm 1,
$$
we conclude that $x - \alpha_{i} y$ is a unit in $\mathbb{Q}(\alpha_{i})$ when $(x , y)$ is a pair of solution to (\ref{1.25}). Suppose that $(x_{1} , y_{1})$ and $(x_{2} , y_{2})$ are two pairs of non-trivial solutions to (\ref{1.25}). Then 
$$
\phi(x_{1} , y_{1}) - \phi(x_{2} , y_{2}) = \left(\log \left|\frac{x_{1} - \alpha_{1}y}{x_{2} - \alpha_{1}y_{2}}\right|,   \ldots , \log \left|\frac{x_{1} - \alpha_{4}y_{1}}{x_{2} - \alpha_{4}y_{2}}\right|\right) = \vec{e}.
$$
Since $ \frac{x_{1} - \alpha_{i}y}{x_{2} - \alpha_{i}y_{2}}$ is a unit in $\mathbb{Q}(\alpha_{i})$, by Proposition \ref{hM5}, we have
$$ 
\| \vec{e} \| = 8 h \left( \frac{x_{1} - \alpha_{i}y}{x_{2} - \alpha_{i}y_{2}}\right) > 2\left(\frac{\log \log 4}{\log 4}\right)^3 > 0.026.
$$
Now we can estimate each side of $\Delta$ from below to conclude that the area of the triangle $\Delta$ is greater than 
$$
\sqrt{3}\left(\frac{\log \log 4}{\log 4}\right)^6 > 0.00029.
$$
Comparing this with (\ref{up5}) we conclude that
$$
2\, \left\| \phi(x_{3} , y_{3}) \right\|  \exp\left(\frac{-\left\| \phi(x_{1} , y_{1}) \right\|}{6}\right) > 0.00029.
$$
 Proposition \ref{exg5} is  immediate from here.
 
Note that when  $\mathbb{Q}(\alpha_{i})$ is a totally real field, we have a better upper bound for the size of $\vec{e}$;
$$
\| \vec{e} \| \geq 4 \log^2 \frac{1 + \sqrt{5}}{2}
$$
(see exercise 2 on page 367 of \cite{Poh5}). Now we can estimate each side of $\Delta$ from below by
$$
4 \sqrt{3} \log^4 \frac{1 + \sqrt{5}}{2}.
$$
 %-------------------------------------------------------------------------------------------------------------------------------
 
% \section{$\phi(x , y)$ as linear forms in logarithms}\label{sec35}
 %-----------------------------------------------------------------------------------------
 
  In order to study the curve $\phi(t)$,  we will consider some well-known  geometric properties of the unit group $U$ of $K = \mathbb{Q}(\alpha)$, where $\alpha$ is a root of $F(x , 1) = 0$. Let $r$ be the number of real conjugate fields of $K$ and $2s$ the number of
 complex conjugate fields of $K$. Then by Dirichlet's unit theorem, the ring of integers $O_{K}$ contains $ r + s - 1$ fundamental units and  there are three possibilities:
  
%\begin{thm}[Dirichlet's  Unit Theorem]
 %Let $K$ be an algebraic number field of degree $n$. Let $r$ be the number of real conjugate fields of $K$ and $2s$ the number of
 %complex conjugate fields of $K$. Then the ring of integers $O_{K}$ contains $ r + s - 1$ fundamental units $\epsilon_{1} , \ldots , 
%\epsilon_{r+s-1}$ such that each unit of $O_{K}$ can be expressed uniquely in the form $u \epsilon_{1}^{n_{1}} \ldots \epsilon_{r+s-1}^
%{n_{r+s-1}}$ , where $u$ is a root of unity in $O_{K}$ and $n_{1}, \ldots , n_{r+s-1}$ are integers.
% \end{thm}
\noindent If $F(x , 1) = 0$ has no real roots, we call $F$ a form of signature $(0 , 2)$.\newline
If $F(x , 1) = 0$ has $4$ real roots, we call $F$ a form of signature $(4 , 0)$.\newline
If $F(x , 1) = 0$ has $2$ real roots, we call $F$ a form of signature $(2 , 1)$.\newline

Here we are working with quartic forms ($r+2s =4$) we have $ r+s-1 = \frac{r}{2} + 1$ fundamental units in $O_{K}$.
Let $\tau$ be the obvious restriction of the embedding of $\mathbb{Q}(\alpha)$ in $\mathbb{C}^{4}$; i.e. $\tau(u) = (u_{1}, u_{2}, u_{3}, u_{4})$, where $u_{i}$ are algebraic conjugates of $u$. 
  By Dirichlet's unit theorem, we have a sequence of mappings
\begin{equation}\label{ta5}
\tau :  U  \longmapsto V  \subset \mathbb{C}^{4}
\end{equation}
and
\begin{equation}\label{lo5}
 \log : V   \longmapsto \Lambda , 
\end{equation}
  where $V$ is the image of the map $\tau$, $\Lambda$ is a $r+s-1$-dimensional lattice, and the mapping $\log$ is defined as follows: \newline
For $(x_{1} , \ldots , x_{r+s-1}) \in  V  $ ,  
 $$
 \log(x_{1} , \ldots , x_{r+s-1}) = (\log|x_{1}|, \ldots,  \log|x_{r+s-1}|),
 $$
where $r$ and $s$ are defined in Dirichlet's unit Theorem. We have $r+s-1\leq 3$.
 If  $(x , y)$ is a pair of solutions to (\ref{1.25}) then 
 $$
 (x - \alpha_{j}y) 
$$
  is a unit in $\mathbb{Q}(\alpha_{i})$.  Suppose that $\mathbb{Q}(\alpha_{i})$ is a real number field and
$$
\lambda_{2}, \ldots ,\lambda_{r+s}
$$
 are  fundamental units of $\mathbb{Q}(\alpha_{i})$ and are chosen so that 
$$
\log\left(\tau(\lambda_{2})\right) , \ldots , \log \left(\tau(\lambda_{r+s})\right)
$$
are respectively first to $r+s-1$ successive minimas of the lattice $\Lambda$. 
%form a \emph{reduced} basis for the lattice $\Lambda$. 
 Let us assume that
  $$
  \left\|\log \left( \tau(\lambda_{2})\right)\right\|  \leq \ldots \leq \left\| \log\left(\tau(\lambda_{r+s})\right)\right\|,
   $$
form a \emph{reduced} basis for the lattice $\Lambda$, so that
 \begin{equation}\label{rep5}
 \phi(x , y) = \phi(1 , 0) + \sum_{k = 2}^{r+s} m_{k}\log\left(\tau(\lambda_{k})\right), \qquad m_{k} \in \mathbb{Z},
 \end{equation}
 with 
  \begin{equation}\label{mk}
  \left\| m_{k} \log \left( \tau(\lambda_{k})\right)\right\|  \leq      \left\| \phi(x , y) - \phi(1 , 0)   \right\|.
  \end{equation}

%------------------------------------------------------------------------------------------------------
\section{ Geometry of $\phi(x , y)$}\label{AA}
%---------------------------------------------------------------------------------------------------

\begin{lemma}\label{sma}
  For every fixed integer $m$, there are at most $2r+2s -2$ solutions $(x , y)$ to (\ref{1.25}) for which  in (\ref{rep5}), $m_{r+s} = m$.
  \end{lemma}
  \begin{proof}
  Let $S$ be the $(r+s-1)$-dimensional affine space of all vectors  
$$\phi(1 , 0) + \sum_{i = 2}^{r+s} \mu_{i}\log\left(\tau(\lambda_{i})\right)\qquad (\mu_{i} \in \mathbb{R}).$$
 Let $\mu_{r+s} = m$. Then the points  
  $$
  \phi(1 , 0) + \sum_{i = 2}^{r+s-1} \mu_{i}\log\left(\tau(\lambda_{i})\right)+ m \log\left(\tau(\lambda_{r+s})\right) 
  $$
  form an $(r+s-2)$-dimensional hyperplane $S_{1}$ of $S$. Put $f(t) = F(t, 1)$.  For $t \in \mathbb{R}$, define  $y(t)$ and $x(t)$ as follows:
\begin{eqnarray*}
y(t)&: =& |f(t)|^{-1/4},\\
x(t)&: = &ty(t).
\end{eqnarray*}
Similar to $\phi(x , y)$, we define the curve $\phi(t)$ on $\mathbb{R}$:
\begin{equation*}
\phi(t) = \left( \phi_{1}(t) , \phi_{2}(t), \phi_{3}(t) , \phi_{4}(t) \right),
\end{equation*}
where, for $1\leq m \leq 4$
\begin{equation*}
\phi_{m}(t) = \log \left| \frac{D^{\frac{1}{4k}}(x(t) - \alpha_{m}y(t))}{\left| f'(\alpha_{m})\right|^{\frac{1}{k} }}\right|.
\end{equation*}
Observe that for an integral solution $(x , y)$ to (\ref{1.25}) and $\phi(x , y)$ defined in (\ref{dfi}), we have
$$
\phi(x , y) = \phi\left(\frac{x}{y}\right).
$$

Let $\vec{N} = (N_{1} , N_{2} , N_{3} , N_{4}) \in S$ be the normal vector of $S_{1}$.
%Let $t = \frac{x}{y}$ and $\phi(t) = \phi(x , y)$. 
Then the number  of times that the curve $\phi(t)$ intersects $S_{1}$ equals the number of solutions in $t$  to 
  \begin{equation}\label{N}
  \vec{N}. \phi(t) = 0.
  \end{equation}
We have
$$
\lim_{t \rightarrow \alpha_{i}^{+} }\log|t - \alpha_{i}| =  -\infty
$$
and
$$
\lim_{t \rightarrow \alpha_{i}^{-} }\log|t - \alpha_{i}| =  -\infty .
$$
   Note that if $\alpha_{i}$ is a non-real root of $F(x , 1)$ then $\bar{\alpha_{i}}$, the complex conjugate of $\alpha_{i}$ is also a root and we have
$$
\log|t - \alpha_{i}| = \log|t - \bar{\alpha_{i}}|.
$$
If $\alpha_{1}, \ldots, \alpha_{r}$ are the reals roots and  $\alpha_{r+1}, \ldots, \alpha_{r+s},\alpha_{r+s+1}, \ldots, \alpha_{r+2s}$ are non-real roots with $\alpha_{r+s+k} = \bar{\alpha}_{r+k}$, then
the derivative $ \frac{d}{dt}  \left( \vec{N}. \phi(t) \right)$ can be written as  $\frac{P(t)}{Q(t)}$,
 where $Q(t) = (t -\alpha_{1}) \ldots  (t -\alpha_{r})(t -\alpha_{r+1}) \ldots  (t -\alpha_{r+s})$ and  $P(t)$ is a polynomial of degree  $r+s -1$. Therefore, the derivative has at most $r+s -1$ zeros and consequently, the equation (\ref{N}) can not have more than $2r+2s -2$ solutions.
   \end{proof}

\begin{lemma}\label{lem100}
Let $F$ be an irreducible monic quartic form. Suppose that $(x , y)$ is a solution to the Thue equation $F(x , y) = \pm 1$ with $y \geq M(F)^{3.5}$. Then
$$
\left\| \phi(1 , 0) \right\| < \left\| \phi(x , y) \right\|.
$$
\end{lemma}
\begin{proof}
Let $\alpha_{1}$, $\ldots$, $\alpha_{n}$ be the roots of $F(z , 1) = 0$. Then 
$$
(\frac{x}{y} - \alpha_{1}) (\frac{x}{y} - \alpha_{2}) (\frac{x}{y} - \alpha_{3})(\frac{x}{y} - \alpha_{4}) = \frac{\pm 1}{y^4}.
$$
There must exist a root $\alpha_{j}$ so that $\left|\frac{x}{y} - \alpha_{j}\right| \geq \frac{1}{y}$. By Lemma \ref{esug} and since $y \geq M(F)^{3.5}$, the absolute value of the term $\phi_{j}(x , y)$ alone exceeds 
$\left\| \phi(1 , 0) \right\| $ (see Lemma \ref{lem10}).
\end{proof}

Recall that we assumed $F(x , y)$ is a monic form. Therefore $(1 , 0)$ is a trivial solution to the equation 
$F(x , y) = \pm 1$.

\bigskip

\textbf{Definition of the set $\mathfrak{A}$}.
 Assume that equation (\ref{1.25}) has more than $2r+2s -2$ solutions. Then we can list $2r+2s-3$ solutions $(x_{i} , y_{i}) \neq (1 , 0)$ ($1 \leq i \leq 2r+2s-3$), so that $r_{i} = \left\| \phi(x_{i} , y_{i})\right\|$ are the smallest among all $\left\| \phi(x , y)\right\|$, where $(x , y)$ varies over all non-trivial pairs of solutions. We denote  the set of all  these $2r+2s -3$ solutions  and the trivial solution $(1 , 0)$ by $\mathfrak{A}$.

  \begin{cor}\label{m}
  Let $(x , y) \not \in \mathfrak{A}$ be a solution to (\ref{1.25}) with $y > M(F)^{3.5}$. Then 
  $$
  \left\|\log \left( \tau(\lambda_{2})\right)\right\|  \leq  \ldots  \leq \left\| \log\left(\tau(\lambda_{r+s})\right)\right\| \leq 2\left\| \phi(x , y)\right\|.
   $$
     \end{cor}
     \begin{proof}
     Since we have assumed that 
$  \left\|\log \left( \tau(\lambda_{2})\right)\right\|  \leq  \ldots  \leq \left\| \log\left(\tau(\lambda_{r+s})\right)\right\|
   $, it is enough to show that   $\left\|\log\left(\tau(\lambda_{r+s})\right)\right\| \leq 2 \left\| \phi(x , y)\right\|$. By Lemma \ref{sma}, there is at least one solution $(x_{0} , y_{0}) \in\mathfrak{A}$ so that 
   $$
   \phi(x , y) - \phi(x_{0} , y_{0}) = \sum_{i = 2}^{r+s} k_{i}\log\left(\tau(\lambda_{i})\right)  , 
  $$ 
   with $k_{n} \neq 0$. Since $\{ \log\left(\tau(\lambda_{i})\right)\}$ is a reduced basis for the lattice $\Lambda$ in (\ref{lo5}), we conclude that 
       $$
       \left\|\log\left(\tau(\lambda_{r+s})\right)\right\| \leq \left\| \phi(x , y) - \phi(x_{0} , y_{0})\right\| \leq  2\left\|\phi(x , y)\right\|.
       $$
   \end{proof}

  \begin{lemma}\label{Dr5}
Suppose  that  $(x , y)$ is a solution to (\ref{1.25}) with $y \geq M(F)^{3.5}$. We have
  $$
   \left\| \phi(x , y) \right\|\geq \frac{1}{2} \log \left(\frac{|D|^{1/12}}{2}\right).
  $$
   \end{lemma}
      \begin{proof}
   Let  $\alpha_{i}$ and $\alpha_{j}$ be two distinct roots of quartic polynomial $F(x , 1)$. We have
\begin{eqnarray*}
\left| e^{\phi_{i}(1 , 0) - \phi_{i}(x , y)} - e^{\phi_{j}(1 , 0) - \phi_{j}(x , y)}\right| & =& \left|\frac{1}{x - y\alpha_{i}} - \frac{1}{x - y\alpha_{j}}\right| \\
& = & \frac{\left|\alpha_{i} - \alpha_{j}\right| |y|}{|x - y\alpha_{i}||x - y\alpha_{j}|}\\
&\geq &\frac{\left|\alpha_{i} - \alpha_{j}\right| }{|x - y\alpha_{i}||x - y\alpha_{j}|}.
\end{eqnarray*}
   Since $|\phi_{i}| < \|\phi\|$ and $\left\| \phi(1, 0) \right\| \leq \left\| \phi(x , y) \right\|$, we may conclude
$$
\left( 2e^{2\left\| \phi(x , y) \right\|}\right)^6 \geq \prod_{1 \leq i < j\leq 4}\left|\frac{1}{x - y\alpha_{i}} - \frac{1}{x - y\alpha_{j}}\right| \geq \sqrt{D}.
$$
   \end{proof}
  %--------------------------------
  %%%%%%%%%%%%%%%%%%%%%%%%%%%%%%%%

%-----------------------------------------------------------------------------------------------------------------------------
%---------------------------------------------------------------------------------------------------------------------------
  \section{ Proof of Proposition \ref{P3} for Forms of signature $(4, 0)$}\label{p3}
  %-------------------------------------------------------------------------------------------------------------
  
  Let $(x , y) \neq (1 , 0)$ be a pair of solution to (\ref{1.25}) and
   \begin{equation*}
 \phi(x , y) = \phi(1 , 0) + \sum_{k = 2}^{r+s} m_{k}\log\left(\tau(\lambda_{k})\right), \qquad m_{k} \in \mathbb{Z}.
 \end{equation*}
Let us set $t = \frac{x}{y}$ and define 
  $$T_{i , j}(t) := \log \left|\frac{(t - \alpha_{i}) (\alpha_{4} - \alpha_{j})}{(t - \alpha_{j}) (\alpha_{4} -\alpha_{i})}\right|.
  $$
 Notice that
  \begin{eqnarray}\label{T5}
  T_{i,j} (x , y) = T_{i,j}(t) & = &  \log \left|\frac{\alpha_{4} - \alpha_{i}}{\alpha_{4} -\alpha_{j}}\right| + \log \left|\frac{t - \alpha_{j}}{t - \alpha_{i} }\right| \nonumber \\ 
  & = &  \log \left|\frac{\alpha_{4} - \alpha_{i}}{\alpha_{4} -\alpha_{j}}\right| + \log \left|\frac{x - \alpha_{j}y}{x - \alpha_{i} y }\right| \nonumber \\ 
  & = &  \log |\lambda_{i,j}| + \sum_{k=2}^{4} m_{i}\log \frac{|\lambda_{k}|} {|\lambda'_{k}|},
  \end{eqnarray}
  where $\lambda_{i,j} =  \log \left|\frac{ \alpha_{4} - \alpha_{i}}{\alpha_{4} -\alpha_{j}} \right| $ and $\lambda_{k}$ and $\lambda'_{k}$ are fundamental units in $\mathbb{Q}(\alpha_{j})$ and $\mathbb{Q}(\alpha_{i})$, respectively. 
   
 \begin{lemma}\label{Tu5}
Let $(x , y)$ be a pair of solution to (\ref{1.25}) with $|y| \geq M(F)^{3.5} $. Then there exists a pair $(i , j)$ for which 
 $$
\left|T_{i, j}(x , y) \right|  <  \exp\left(-\frac{ \left\| \phi(t) \right\|}{6}\right).$$
  \end{lemma}
\begin{proof}
This is a consequence of Lemma \ref{gp15} and the proof goes exactly the same as the proof of
 Lemma 6.8 in \cite{AO5}.
\end{proof}

  Let index $\sigma $ be the isomorphism  from 
$\mathbb{Q}(\alpha_{i})$ to $\mathbb{Q}(\alpha_{j})$ such that $\sigma(\alpha_{i}) = \alpha_{j}$. We may assume that $\sigma(\lambda_{i}) = \lambda'_{i}$ for $i = 2, 3, 4$.  Let $(x_{1}, y_{1})$ , $(x_{2} , y_{2})$ , $(x_{3} , y_{3})$ be three distinct  solutions to (\ref{1.25})  with 
 $$
 y_{k} \geq M(F)^{3.5}
 $$
 and
 $$
\left| x_{k} - \alpha_{4} y_{k}\right|  = \min_{1\leq i \leq 4} \left| x_{k} - \alpha_{i} y_{k}\right|   \quad k \in \{1 , 2 , 3 \}.
$$
This assumption will lead us to a contradiction at the end of this section, implying that related to each real root of $F(x, 1) = 0$, there are at most $2$ solutions with $y \geq M(F)^{3.5}$.  Recall that related to a non-real root, there exists no such solution. 

 Put $r_{k} = \left\| \phi( x_{k} , y_{k}) \right\|$ and assume that 
$r_{1} \leq r_{2} \leq r_{3}$.
 We will apply  Matveev's lower bound to 
 $$
  T_{i,j}(x_{3} , y_{3}) =   \log |\lambda_{i,j}| + \sum_{k=2}^{r+s} m_{k}\log \frac{ |\lambda_{k}|} { |\lambda'_{k}|} , $$
where $(i , j)$ is chosen so that  Lemma \ref{Tu5} is satisfied and  $m_{k} \in \mathbb{Z}$. Forms of signature $(0 , 2)$ have one fundamental unit and forms of signature $(2 , 1)$ have two fundamental units.
 Moreover,  if $ \lambda_{i,j}$ is a unit then we can write $T_{i,j}(x , y) $ as a linear form in fewer number of  logarithms.  We remark here that when dealing with linear forms in $2$ logarithms, one can use sharper lower bounds (see for example \cite{Gou5}). 

Suppose that $\lambda$ is a unit in the number field and $\lambda'$ is its algebraic conjugate. We have
$$
h(\lambda') = h(\lambda)  = \frac{1}{8}\left|\log\left(\tau(\lambda)\right)\right|_{1} , 
$$
where $h$ is the logarithmic height and  $| \  |_{1}$ is the $L_{1}$ norm on $\mathbb{R}^{4}$ and the mappings $\tau$ and $\log$ are defined in (\ref{ta5}) and (\ref{lo5}) .  So we have
$$
h(\lambda) = \frac{1}{8}\left|\log\left(\tau(\lambda)\right)\right|_{1} \leq \frac{\sqrt{4}}{8} \left \| \log\left(\tau(\lambda)\right)\right\| ,
$$
where $\| \|$ is the $L_{2}$ norm on $\mathbb{R}^{r+s-1}$.  Since $\alpha_{4}$, $\alpha_{i}$ and $\alpha_{j}$ have degree $4$ over $\mathbb{Q}$, the number field $\mathbb{Q}(\alpha_{4}, \alpha_{i}, \alpha_{j})$ has degree $d \leq 24$ over $\mathbb{Q}$.
So when $\lambda$ is a unit
\begin{equation}\label{h15}
\max \{dh(\frac{\lambda}{\lambda'}) , \left|\log(\left|\frac{\lambda}{\lambda'}\right|)\right| \} \leq 
\max \{24 h(\frac{\lambda}{\lambda'}) , |\log(\left|\frac{\lambda}{\lambda'}\right|)| \} \leq 12 \left\|\log\left(\tau(\lambda)\right)\right\|.
\end{equation}
In order  to apply  Theorem \ref{mat5} to $T_{i,j}(x , y)$, we will take , for $k> 1$,
$$
A_{k} =  12 \left\|\log\left(\tau(\lambda)\right)\right\|.
$$
By (\ref{mk}) and Lemma \ref{lem100},  we have
 $$
  A_{k} \leq 24 \|\phi(x,y)\|,  \   \,    \textrm{for} \ \    k>1 .
 $$

Now we need to estimate $A_{1}$.
\begin{lemma}\label{EstimateOfHeightDelta}
Let $F$ be an irreducible  binary quartic form with integral coefficients.
Assume that $(x,y)$ is a solution to \textup{(\ref{1.25})} with $y \geq M(F)^{3.5}$.
Then, we have
\begin{equation*}
	h\left(\frac{\alpha_k-\alpha_i}{\alpha_k-\alpha_j}\right)
	\le
	2\log 2 + 2\|\phi(x,y)\|.
\end{equation*}
\end{lemma}
\begin{proof}
Let  $\beta_i = x-y\alpha_i$.
We have
\begin{equation*}
	\frac{\alpha_k-\alpha_i}{\alpha_k-\alpha_j}
	=
	\frac{\beta_k-\beta_i}{\beta_k-\beta_j}.
\end{equation*}
Thus by (\ref{bj15}) and (\ref{bj25}),
\begin{equation*}
	h\left(\frac{\alpha_k-\alpha_i}{\alpha_k-\alpha_j}\right)
	\le 2\log 2 + 4h(\beta_k).
\end{equation*}
To complete the proof, we will show that 
$$
h(\beta_{k})\leq \frac{1}{2}\|\phi(x,y)\|.
$$
Set $$v_i = \log|\beta_i| = \phi_i(x,y)-\phi_i(1,0)\qquad \textrm{for}\qquad   i\in \{1,2,3, 4\}$$ and
$$\vec{v}=(v_1,v_2,v_3,v_4).$$
Since $\beta_k$ is a unit, we have
\begin{equation*}
	h(\beta_k) = \frac{1}{8}\sum_{i=1}^4 \left|v_i\right|
		= \frac{1}{8} (s_1, s_2, s_3 , s_4) \cdot \vec{v},
\end{equation*}
where $s_1, s_2, s_3, s_4\in\{+1, -1\}$.
Since $\|(s_1, s_2, s_3, s_4)\| = 2$, we obtain
\begin{equation*}
	h(\beta_k) \le \frac{1}{4}\|\vec{v}\|.
\end{equation*}
On the other hand,  we have
\begin{equation*}\label{EstimateOfV}
	\|\vec{v}\| \le \|\phi(x,y)\| + \|\phi(1,0)\|
		\le 2\|\phi(x,y)\|.
\end{equation*}
This completes our proof.
\end{proof}

Set, for $k \in \{1, 2, 3\}$,
$$
r_{k} = \|\phi(x_k, y_k)\|.
$$
We may take
$$
A_{1}  = 48\log 2+ 48 r_{1}
$$
(recall that $\alpha_{1}$, $\alpha_{i}$ , $\alpha_{j}$ are algebraic conjugates).
Since $m_{1} =1$, we will put
$$B = \max\{1 , \max\{m_{j}A_{j}/A_{1}: \  1\leq j  \leq 4\}\}.$$
 Since we assumed that $\tau(\lambda_{i})$, $2 \leq i \leq r+s-1$ are successive minimas for the lattice $\Lambda$, we have
$$
\left| m_{j}A_{j}\right| \leq r_{3} + \| \phi(1 , 0) \|  < 2 r_{3}.
$$
 Thus we may take  $B = \frac{ r_{3}}{12}$ (see (\ref{mk})).

Proposition \ref{mat5} implies that for a constant number $K$,
$$
\log T_{i, j} (x_{3} , y_{3}) > - K \,  r_{1}^4 \log r_{3}.
$$
Comparing this with Lemma \ref{Tu5}, we have
$$
 \left(\frac{-r_{3}}{6}\right)> -K\,   r_{1}^{4}\log r_{3},
$$
or
$$
 \frac{r_{3}} {\log r_{3}} < 6 K \, r_{1}^{4} .
$$
Thus there is a computable constant number $K_{1}$, so that
\begin{equation}\label{r55}
 r_{3} <  K_{1} \, r_{1}^{4} ,
\end{equation}
This is because $r_{3}$ is large enough  by Lemma \ref{Dr5}.
But by Lemma \ref{exg5} we have
$$
r_{3} > 0.00014 \exp\left(\frac{r_{1}}{6}\right).
$$
This is  a contradiction, for by Lemma \ref{Dr5},
$$
r_{1} \geq \frac{1}{2} \log \left( \frac{|D|^{\frac{1}{12}}}{2}\right)
$$
and $D$ is large. 
Thus, there are at most $2$ solutions $(x , y)$ with $y \geq M(F)^{3.5}$  related to each real  root $\alpha_{i}$. The proof of Proposition \ref{P3} for forms with signature $(4 , 0)$ is complete now.  This argument can be used for forms with signature $(2 , 1)$, as well and will give us an  bound of $17$ upon the number of solutions to (\ref{1.25}) for this case. In the next section, we will see that for forms of signature $(2 , 1)$ we do not need to consider the set $\mathfrak{A}$, as the lattice generated by the fundamental units of the corresponding number field is contained in a plane.

  %--------------------------------------------------------------------------------------------------------------
  \section{ Proof of Proposition \ref{P3} for Forms of signature $(2, 1)$}\label{p4}
  %-------------------------------------------------------------------------------------------------------------------
  
Let $F(x , y)$ be a quartic form of signature $(2 , 1)$.  Suppose that  $(x_{1} , y_{1})$, $(x_{2} , y_{2})$ and $(x_{3} , y_{3})$ are three pairs of non-trivial solution to (\ref{1.25}) with 
$$
\left| x_{j} - \alpha_{4} y_{j}\right| < 1,
$$
and $|y_{j}| >M(F)^{3.5} $,
for $j \in \{1 , 2 , 3\}$.         
 We will assume $\alpha_{4}$ is real (recall that related to a non-real root, there exists no  solution $(x , y)$ with $|y| >M(F)^{3.5} $).
The real number field 
 $\mathbb{Q}(\alpha_{4})$ has  two fundamental units
$\lambda_{2}, \lambda_{3}$
 chosen so that 
$\log\left(\tau(\lambda_{2})\right)  , \log \left(\tau(\lambda_{3})\right)$
are  successive minimas of the lattice $\Lambda$, with
\begin{equation}\label{paral}
  \left\|\log \left( \tau(\lambda_{2})\right)\right\|   \left\| \log\left(\tau(\lambda_{3})\right)\right\| \geq \frac{2}{\sqrt{3}}\textrm{Vol}(\Lambda),
   \end{equation}
where $\textrm{Vol} (\Lambda)$  is the volume of fundamental parallelepiped of lattice $\Lambda$.   
If $(x , y)$ is a  solution to $|F(x , y)| = 1$, then 
$$
\phi(x , y)  \in \phi(1 , 0) + \Lambda = \Lambda_{1}.
$$   
Note that 
$$
\textrm{Vol}(\Lambda) = \textrm {Vol}(\Lambda_{1}).
$$

  For distinct pairs of solution  $(x_{1} , y_{1})$, $(x_{2} , y_{2})$ and $(x_{3} , y_{3})$ , three vectors  $\phi(x_1 , y_1)$,    $\phi(x_2 , y_2)$ and    $\phi(x_3 , y_3)$       generate a sub-lattice of $\Lambda_{1}$ with the volume of fundamental parallelepiped equal to $2A$. 
   Therefore,
   \begin{equation}\label{AV}
   2 A \geq \textrm{Vol} (\Lambda_{1}) = \textrm{Vol} (\Lambda).
   \end{equation}
   On the other hand, by (\ref{up5}), $A$, the area of $\Delta$  is less than 
\begin{equation*}
2   \left\| \phi(x_{3} , y_{3}) \right\| \exp\left(\frac{-\left\| \phi(x_{1} , y_{1}) \right\|}{6}\right).
\end{equation*}
This, together with (\ref{AV}) gives
\begin{equation}\label{exg5alt}
\left\| \phi(x_{3} , y_{3}) \right\| > \frac{\textrm{Vol} (\Lambda)}{4}\exp\left(\frac{\left\| \phi(x_{1} , y_{1}) \right\|}{6}\right).
\end{equation}
Let us replace  Proposition \ref{exg5} by the above inequality when $F(x , y)$ has signature $(2, 1)$.
   %%%%%%%%%%%%%%%%%%%%%%%%%%%%%%%%%%%%%%%%%%
 %%%%%%%%%From Previous Section TO BE EDITED AND ADJUSTED

 Put $r_{k} = \left\| \phi( x_{k} , y_{k}) \right\|$ and assume that 
$r_{1} \leq r_{2} \leq r_{3}$.
 We will apply  Matveev's lower bound to 
 $$
  T_{i,j}(x_{3} , y_{3}) =   \log |\lambda_{i,j}| + \sum_{k=2}^{r+s} m_{k}\log \frac{ |\lambda_{k}|} { |\lambda'_{k}|} , $$
where $(i , j)$ is chosen so that  Lemma \ref{Tu5} is satisfied and  $m_{k} \in \mathbb{Z}$.  Similar to Section \ref{p3}, let index $\sigma $ be the isomorphism  from 
$\mathbb{Q}(\alpha_{i})$ to $\mathbb{Q}(\alpha_{j})$ such that $\sigma(\alpha_{i}) = \alpha_{j}$. We may assume that $\sigma(\lambda_{i}) = \lambda'_{i}$ for $i = 2, 3, 4$.    Recall that related to a non-real root, there exists no such solution. 

 We use our estimation from Section \ref{p3} for $A_{1}$ and $B$:
$$
A_{1}  = 48\log 2+ 48 r_{1}
$$
and
$$
B = \frac{ r_{3}}{12}.
$$

 For fundamental units  $\lambda_{k}$, with $k \in \{2 , 3\}$, using  (\ref{h15}), we may put
$$
A_{k} =  12 \left\|\log\left(\tau(\lambda)\right)\right\|.
$$

Theorem \ref{mat5} and (\ref{paral})  imply  that for a constant number $K$,
$$
\log T_{i, j} (x_{3} , y_{3}) > - K \,  r_{1}  \log r_{3} \textrm{Vol} (\Lambda).
$$
Comparing this with Lemma \ref{Tu5}, we have
$$
 \left(\frac{-r_{3}}{6}\right)> -K\,    r_{1}  \log r_{3} \textrm{Vol} (\Lambda),
$$
or
$$
 \frac{r_{3}} {\log r_{3}} < 6 K \,  r_{1}  \textrm{Vol} (\Lambda) .
$$
Thus, since $r_{3}$ is large enough  by Lemma \ref{Dr5}, there is a computable constant number $K_{2}$, so that
\begin{equation*}
 r_{3} <  K_{2} \,  r_{1}  \textrm{Vol} (\Lambda)
\end{equation*}
(compare this with (\ref{r55})). 
But by ( \ref{exg5alt}) we have
$$
r_{3} >   \frac{\textrm{Vol} (\Lambda)}{4} \exp\left(\frac{r_{1}}{6}\right).
$$
Consequently, for some positive constant $K_{3}$,
$$ 
r_{1} > K_{3}\,  \exp\left(\frac{r_{1}}{6}\right).
$$
This is  a contradiction by Lemma \ref{Dr5} and since the discriminant is large.
Thus, there are at most $2$ solutions $(x , y)$ with $y \geq M(F)^{3.5}$  related to each real  root $\alpha_{i}$. The proof of Proposition \ref{P3} and  Theorem \ref{main5} are complete now. 

\section{Acknowledgements} 
I would like to thank Professor Ryotaro Okazaki for his helpful comments  and suggestions.  
 Most of this work has been done during my visit to mathematisches forschungsinstitut oberwolfach and   Max-Planck institute for mathematics in Bonn.  I am very grateful to both institutes for support and hospitality.

%--------------------------------------------------------------------------------------
%--------------------------------------------------------------------------------------

%----------------------------------------------------------------------------------------------------------------------------------

\end{document}